\documentclass[11pt]{amsart}

\usepackage[hidelinks]{hyperref}
\usepackage{amsmath}
\usepackage{amstext}
\usepackage{amssymb}
\usepackage{amsthm}
\usepackage{enumerate}
\usepackage{bbm}
\usepackage{comment}
\usepackage[USenglish]{babel}
\usepackage[utf8]{inputenc}
\usepackage[T1]{fontenc}
\usepackage{xcolor}
\usepackage{tikz}
\usetikzlibrary{arrows}

\makeatletter
\def\imod#1{\allowbreak\mkern10mu({\operator@font mod}\,\,#1)}
\makeatother

\theoremstyle{plain}
\newtheorem{thm}{Theorem}
\theoremstyle{definition}
\newtheorem{defi}[thm]{Definition}
\newtheorem{cor}[thm]{Corollary}
\newtheorem*{cor*}{Corollary}
\newtheorem*{lem*}{Lemma}
\newtheorem*{thm*}{Theorem}
\newtheorem*{clm*}{Claim}
\newtheorem{lem}[thm]{Lemma}
\newtheorem{question}[thm]{Question}
\newtheorem{prop}[thm]{Proposition}

\newtheorem{rem}[thm]{Remark}

\newtheorem*{acknowledgement*}{Acknowledgement}

\DeclareMathOperator{\dom}{dom}
\DeclareMathOperator{\ran}{ran}

\DeclareMathOperator{\id}{id}

\DeclareMathOperator{\ZF}{ZF}

\DeclareMathOperator{\CH}{CH}

\renewcommand{\id}{\mathrm{id}}

\newcommand{\bbP}{\mathbb{P}}
\newcommand{\bbQ}{\mathbb{Q}}

\newcommand{\bbZ}{\mathbb{Z}}

\renewcommand{\phi}{\varphi}

\newcommand{\om}{\omega}

\newcommand{\frc}{\mathfrak{c}}

\newcommand{\calA}{\mathcal{A}}

\newcommand{\calE}{\ensuremath{\mathcal{E}}}

\newcommand{\calG}{\ensuremath{\mathcal{G}}}

\newcommand{\calI}{\ensuremath{\mathcal{I}}}

\newcommand{\calO}{\ensuremath{\mathcal{O}}}
\newcommand{\calP}{\mathcal{P}}

\newcommand{\forces}{\Vdash}

\newcommand{\restr}{\!\upharpoonright\!}

\newcommand{\Exists}{\,\exists}

\newcommand{\concatB}{\mathbin{\rotatebox[origin=c]{90}{\scalebox{.7}{(\kern1ex)}}}}

\newcommand{\extends}{\mathbin{\leq}}
\DeclareMathOperator{\hull}{hull}

\newcommand{\simpleset}[1]{{\{{#1}\}}}
\newcommand{\simpleseq}[1]{{\langle{#1}\rangle}}
\newcommand{\set}[2]{{\{ {#1} \mid {#2} \}}}
\newcommand{\seq}[2]{{\langle {#1} \mid {#2} \rangle}}

\definecolor{dodger}{rgb}{0.0,0.5,1.0}

    \DeclareMathOperator{\fix}{fix}
    \DeclareMathOperator{\cofin}{cofin}

\date{}
\begin{document}


\title{Tight cofinitary groups}


\author{Vera Fischer}
\address{University of Vienna, Institute for Mathematics, Kolingasse 14-16, 1090 Vienna, Austria}
\email{vera.fischer@univie.ac.at}


\author{L. Schembecker}
\address{Department of Mathematics, University of Hamburg, Bundesstrasse 55, 20146 Hamburg, Germany}
\email{lukas.schembecker@uni-hamburg.de}


\author{David Schrittesser}
\address{Institute for Advanced Studies in Mathematics, Harbin Institute of Technology, 92 West Dazhi Street, Nangang District, Harbin, Heilongjiang Province, China
150001}
\email{david.schrittesser@univie.ac.at}

\thanks{The first and second authors would like to thank the Austrian Science Fund FWF for the generous support through grant number START Y1012-N35 [10.55776/Y1012].}

\makeatletter
\@namedef{subjclassname@2020}{%
  \textup{2020} Mathematics Subject Classification}
\makeatother




\begin{abstract} 
We introduce the notion of a tight cofinitary group, which captures forcing indestructibility of maximal cofinitary groups for a long list of partial orders, including Cohen, Sacks, Miller, Miller partition forcing and Shelah's poset for diagonalizing maximal ideals. Introducing a new robust coding technique, we  establish the relative consistency of $\mathfrak{a}_g=\mathfrak{d}<\mathfrak{c}=\aleph_2$ alongside the existence of a $\Delta^1_3$-well-order of the reals and a co-analytic witness for $\mathfrak{a}_g$.
\end{abstract}

\maketitle

\section{Introduction} \label{SEC_Introduction}

{\emph{Cofinitary groups}} are subgroups of $S_\infty$, the group of all permutations of the natural numbers, which have the property that all their non-identity elements have only finitely many fixed points. Throughout the paper, we denote by $\text{cofin}(S_\infty)$ the collection of all such permutations. That is for $f\in S_\infty$, $f\in \text{cofin}(S_\infty)$ iff the set 
$\fix(f):=\{n\in\omega: f(n)=n\}$ is finite. {\emph{Maximal cofinitary groups}} are cofiniaty groups which are maximal under inclusion. In the 80's they were of interest to algebraists, like Adeleke,  Cameron and Neumann (see \cite{SA}, \cite{PC} and \cite{PN}), who could show that a countable cofinitary group is never maximal, even more a bounded cofinitary group is never maximal, and that there always is a cofinitary group of cardinality $\mathfrak{c}$. In \cite{YZhang} Y. Zhang developed a ccc forcing notion  leading to the consistency of $ \aleph_1<\mathfrak{a}_g<\mathfrak{c}$, where $\mathfrak{a}_g$ denotes the least cardinality of a maximal cofinitary group.  An interesting poset adjoining a maximal cofinitary group of arbitrary uncountable size, which also plays a crucial role in the consistency of $\mathfrak{a}_g=\aleph_\omega$ can be found in \cite{VFAT}. An excellent exposition of the general properties of most of the classical combinatorial cardinal characteristics can be found in \cite{Blass}.

The study of the projective complexity of maximal cofinitary groups has already a comparatively long history.  In 2008~\cite{SGYZ} Gao and Zhang show that in the constructible universe $L$ there is a maximal cofinitary group with a co-analytic set of generators. The result was improved a year later by Kastermans~\cite{BAK1} who constructed a co-analytic maximal cofinitary group (in $L$).  In both of those results the existence of a $\Sigma^1_2$ definable projective well-order of the reals,  as well as the Continuum Hypothesis, play a crucial role, thus leaving aside the problem of providing projective maximal cofinitary groups of low projective complexity in models of large continuum. 
 
The study of models of $\neg\CH$ with a projective well-order on the reals has a much longer history, initiated by the work of Leo Harrington~\cite{LeoH}, in which he obtained a model of $\mathfrak{c}=\aleph_2$ with a $\mathbf{\Delta^1_3}$-well-order of the reals. In~\cite{VFSF}, the first author of the paper and S. D. Friedman, introduced the method of coding with perfect trees to obtain models of $\mathfrak{c}=\aleph_2$, a light-face $\Delta^1_3$-well-order of the reals and various cardinal characteristics constellations. Note that by a result of Mansfield if there is a $\Sigma^1_2$-well-order of the reals, then every real is constructible and so the existence of a $\Delta^1_3$-well-order on the reals is optimal for models of $\mathfrak{c}>\aleph_1$.  The well-order can be used to produce $\Sigma^1_3$-definable sets of reals of interest of cardinality $\mathfrak{c}$, including maximal cofinitary groups. But the complexity of the maximal cofinitary group thus obtained is not lowest possible: Indeed, Horowitz and Shelah in~\cite{HS1} show that there is always a Borel maximal cofinitary group (which is necessarily of cardinality $\mathfrak{c}$).  In~\cite{VFDSAT}, the first and third author of the current paper jointly with A. Törnquist construct a co-analytic maximal cofinitary group, which is Cohen indestructible. Thus, they obtain a model in which $\mathfrak{a}_g=\mathfrak{b}<\mathfrak{d}=\mathfrak{c}=\kappa$ where $\kappa$ can be an arbitrarily large cardinal of uncountable cofinality and $\mathfrak{a}_g$ is witnessed by a maximal cofinitary group of optimal projective complexity. Providing a model of $\mathfrak{d}=\mathfrak{a}_g=\aleph_1<\mathfrak{c}$ with an optimal projective witness to $\mathfrak{a}_g$ remained open until the current paper and required significantly different ideas. 
  
Recall 
that a family of reals $\calE\subseteq{^\omega\omega}$ is said to be {\emph{eventually different}} if for any distinct elements of $f,g$ from $\calE$ there is $n\in\omega$ such that $\forall k\geq n (g(k)\neq f(k))$. An eventually different family is maximal, abbreviated m.e.d., if it is maximal under inclusion and the minimal cardinality of a maximal eventually different family is denoted $\mathfrak{a}_e$. Every cofinitary group is an eventually different family
and the graphs of the elements of every eventually different family form an almost disjoint family (in $\omega\times\omega$).\footnote{For work regarding ideals other than the ideal of finite sets, see \cite{HE}.} 

Inspired by the construction of a tight almost disjoint family, in~\cite{VFCS}, Fischer and Switzer introduce the notion of a tight eventually different family of reals. Tight eventually different families are maximal in a strong sense, as they are indestructible by a long list of partial orders, including Cohen, Sacks (its products and countable support iterations), Miller, Miller partition forcing and others. They are never analytic, exist under CH and $\mathsf{MA}(\sigma\text{-centered})$, and are used in~\cite{VFCS} to provide co-analytic witnesses to $\mathfrak{a}_e=\aleph_1$ in various forcing extensions. Moreover, in~\cite{JBVFCS} the authors show that Sacks coding preserves (in a strong sense) co-analytic tight eventually different families leading to models of $\mathfrak{c}=\aleph_2$ with a $\Delta^1_3$-well-order of the reals and co-analytic witnesses of $\mathfrak{a}_e=\aleph_1$. However, the consistency of the existence of a co-analytic maximal cofinitary group of size $\aleph_1$ in the presence of a $\Delta^1_3$-well-order of the reals and $\mathfrak{c}=\aleph_2$ remained open until the current paper.

The paper is structured as follows: In section 2 we introduce the notion of a tight cofinitary group. Tight cofinitary groups are maximal in a strong sense, as they are indestructible by a long list of partial orders.  In section 3 we introduce a slight modification of Zhang's poset for adding a new generator to a cofinitary group and show that under $\mathsf{MA}(\sigma\text{-centered})$ tight cofinitary groups exist and moreover that every cofinitary group of size $<\mathfrak{c}$ is contained in a tight cofinitary group of size $\mathfrak{c}$. In Section 4 we introduce a robust new coding technique allowing us to code reals into the lengths of orbits of every new word.
Crucially, compared to other coding techniques for cofinitary groups (i.e.\ as in \cite{VFDS}) our new coding is parameter-free and hence may be applied to groups of uncountable size.
Furthermore, as we code into orbits rather than actual function values, a more general generic hitting lemma (see Lemma~\ref{LEM_Tree_Extension}) required for tightness holds.
As an application of this new coding, we obtain our first main result, namely the existence of a co-analytic tight cofinitary group in $L$ (see Theorem \ref{THM_Coanalytic_Group}).  In Section 5,  we obtain our second main result, proving that Sacks coding strongly preserves the tightness of nicely definable cofinitary groups (see Theorem \ref{main_two}) and so establishing the consistency of $\mathfrak{a}_g=\mathfrak{d}=\aleph_1<\mathfrak{c}=\aleph_2$ alongside a tight co-analytic witness to $\mathfrak{a}_g$ and a $\Delta^1_3$-well order of the reals.

Our results, together with earlier investigation into the existence of nicely definable, forcing indestructible combinatorial sets of reals of interest, imply that each of the constellations 
$\mathfrak{a}_g=\mathfrak{u}=\mathfrak{i}=\aleph_1<\mathfrak{c}=\aleph_2$, $\mathfrak{a}_g=\mathfrak{u}=\aleph_1<\mathfrak{i}=\mathfrak{c}=\aleph_2$, $\mathfrak{a}_g=\mathfrak{i}=\aleph_1<\mathfrak{u}=\mathfrak{c}=\aleph_2$, $\mathfrak{a}_g=\aleph_1<\mathfrak{i}=\mathfrak{u}=\aleph_2$ can hold alongside a $\Delta^1_3$-well-order of the reals, a co-analytic tight cofinitary group of cardinality $\aleph_1$, as well as a co-analytic filter base of cardinality $\aleph_1$ in the first and second constellations, and a co-analytic selective independent family in the first and third constellations. We conclude the paper with listing some interesting remaining open questions.

In particular, we give a new proof of the existence of a Miller indestructible maximal cofinitary group (originally obtained in~\cite{BKYZ} using a diamond sequence) and answer Question 2 of~\cite{VFCS}.

\section{Tight cofinitary groups} \label{SEC_Tight}

In \cite{VFCS} Switzer and the first author defined a notion of tightness for eventually different families of permutations.
As we will see, the same notion of tightness may be used in the context of cofinitary groups. In this section we will briefly recall the definitions of \cite{VFCS}, which are important for the current work.

A tree $T \subseteq \omega^{<\omega}$ is called injective iff every element of $T$ is injective.
Let $\calP$ be a maximal eventually different family of permutations.
That means that for every permutation $f \in S_\infty$ there is a $g \in \calP$ such that $f$ and $g$ are infinitely often equal.
Now, consider the tree $T$ with only branch $f$.
Clearly this is an injective tree as $f$ is injective, and we may rephrase the relationship between $f$ and $g$ in terms of $T$ and $g$ by: for every $s \in T$, there is a branch through $T$ which is infinitely often equal to $g$.
Hence, we define

\begin{defi} \label{DEF_Densely_Diagonalizing}
    Let $T \subseteq \omega^{<\omega}$ be an injective tree and $g \in \omega^\omega$.
    Then, we say $g$ {\emph{densely diagonalizes}} $T$ iff for all $s \in T$ there is a $t \in T$ with $s \subseteq t$ and $k \in \dom(t) \setminus \dom(s)$ such that $t(k) = g(k)$.
\end{defi}

Generally, the idea of tightness is to strengthen the maximality of $\calP$ by requiring that for all injective trees $T$ there is a $g \in \calP$ densely diagonalizing $T$ - in fact for countably many injective trees there should be just one $g \in \calP$ which densely diagonalizes all of them (see Definition~\ref{DEF_Tight_Permutations}).
However, this cannot quite work, as for example for the injective tree $T$ with only branches $f_1 \neq f_2 \in \calP$, there cannot be a single $g \in \calP$ diagonalizing $T$ (Since $\calP$ is eventually different, $g$ would have to be both equal to $f_1$ and $f_2$).

Thus, for a family $\calP \subseteq \omega^{\omega}$ of eventually different permutations we require diagonalization only for the following set of injective trees positive with respect to $\calP$:

\begin{defi} \label{DEF_Injective_Tree_Ideal}
    The {\emph{injective tree ideal}} generated by $\calP$, denoted $\calI_{i}(\calP)$, is the set of all injective trees $T \subseteq \omega^{<\omega}$ such that there is an $s \in T$ and a finite set $\calP_0 \subseteq \calP$, so that for all $t \in T$ with $s \subseteq t$ and $k \in \dom(t) \setminus \dom(s)$ there is an $f \in \calP_0$ with $t(k) = f(k)$.
    Dually, we denote with $\calI_i(\calP)^+$ all injective trees $T \subseteq \omega^{<\omega}$ not in $\calI_i(\calP)$.
\end{defi}

In other words, a tree $T$ is an element of $\calI_i(\calP)^+$ iff for all $s \in T$, the restriction $T_s$ of $T$ to $s$ cannot be covered by the graphs of finitely many members of $\calP$.
Note, that despite its name $\calI_{i}(\calP)$ is not an ideal nor does it generate one.
The purpose of this naming is its analogical role as the associated ideal for tight mad families (see \cite{GHT}).
There, a subset of $\om$ is positive with respect to a mad family $\calA$ iff it cannot be almost covered by the union of finitely many members of $\calA$.

Remember, that tightness should imply some kind of forcing indestructibility, so we also have the following intuition as to why we consider injective trees.
Given a forcing $\bbP$, $p \in \bbP$ and a $\bbP$-name $\dot{f}$ for a permutation which is eventually different different from $\calP$, we may consider the hull of $\dot{f}$ below $p$, which is the set of all possible finite approximations of $\dot{f}$ below $p$.
\[
    \text{hull}(\dot{f}, p) = \set{s \in \omega^{<\omega}}{\exists q \extends p \text{ such that } q \forces s \subseteq \dot{f}}
\]
It is easy to see, that by assumption on $\dot{f}$ we have $\hull(\dot{f},p) \in \calI_i(\calP)^+$.
So in this context, we essentially require that for every hull of a name for a real destroying the maximality of $\calP$, there is a $g \in \calP$ diagonalizing this hull.
It turns out, that this property already implies forcing indestructibility by many forcings $\bbP$, which is in some sense the main result of \cite{VFCS}.

\begin{defi} \label{DEF_Tight_Permutations}
    An eventually different family of permutations $\calP$ is called {\emph{tight}} iff for all sequences $\seq{T_n \in \calI_i(\calP)^+}{n < \omega}$ there is a $g \in \calP$ such that $g$ diagonalizes $T_n$ for every $n < \omega$.
\end{defi}

In our setting, since every cofinitary group is a family of eventually different permutations we may define tightness in terms of eventually different permutations.

\begin{defi} \label{DEF_Tight_Cofinitary}
    Let $\calG$ be a cofinitary group.
    Then we say $\calG$ is {\emph{tight}} iff $\calG$ is tight as an eventually different family of permutations.
\end{defi}

\begin{rem} \label{DEF_Tight_Implies_Maximal}
    Since tight eventually different families of permutations are maximal eventually different (m.e.d.) families of permutations (cf.\ Proposition~5.3 in \cite{VFCS}), also tight cofinitary groups are maximal cofinitary groups as we have the following implication:
    $$
        \calG \text{ group and m.e.d.\ family of permutations } \Rightarrow \calG \text{ maximal cofinitary group}
    $$
    To the knowledge of the authors the reverse implication is not known even if the maximality on the left side is restricted to elements in $\cofin(S_\infty)$.
\end{rem}

\section{Existence of tight cofinitary groups under Martin's Axiom} \label{SEC_Zhang}

Similarly to Theorem~5.4 in \cite{VFCS} for eventually different families of permutations, we verify that ${\sf MA}(\sigma\text{-centered})$ implies the existence of tight cofinitary groups. The conditions of the poset give finite approximations to the new generator, while the extension relation guarantees not only that the new generator is indeed a cofinitary permutation, but also that every new permutation in the group resulting from adjoining this new generator to the given cofinitary group, has only finitely many fixed points. Our notation regarding words, partial evaluations and fixed point of such evaluations follow \cite{VFAT}, \cite{VFLS}, \cite{VFDSAT} and \cite{YZhang}. 
Following \cite{VFLS} we will use a version of Zhang's forcing (see \cite[Definition 2.1]{YZhang}) which uses nice words.

\begin{defi}\label{DEF_Words_And_Evaluation}
    Let $\calG$ be a cofinitary group.
    \begin{enumerate}
    \item We denote with {\emph{$W_\calG$}} the set of all words in the language $\calG \cup \simpleset{x, x^{-1}}$, where $x$ is treated as a new symbol.
    \item 
    Given any (partial) injection $s$ and $w \in W_\calG$ we denote with {\emph{$w[s]$}} the (partial) injection, where in $w$, $x$ and $x^{-1}$ are replaced by $s$ and $s^{-1}$ and then the corresponding (partial) functions are concatenated.
    \item 
    Finally, for any $w \in W_\calG$ we denote with {\emph{ $w \restr \calG$}} the finite set of all $g \in \calG$ such that $g$ or $g^{-1}$ appears as a letter in $w$ or $g = \id$.
    Similarly, for $E \subseteq W_\calG$ we set {\emph{$E \restr \calG := \bigcup_{w \in E} w \restr \calG$}}.
    \end{enumerate}
\end{defi}

\begin{defi}\label{DEF_Nice_Words}
    Let $\calG$ be a cofinitary group.
    We denote with $W^*_\calG$ the following subset of $W_\calG$, which we call {\emph{the set of nice words}}.
    A {\emph{word $w \in W_\calG$ is nice}} iff $w = x^k$ for some $k > 0$ or there are $k_0 > 0$ and $k_1, \dots, k_l \in \bbZ \setminus \simpleset{0}$ and $g_0, \dots ,g_l \in \calG \setminus \simpleset{\id}$ such that
    $$
        w = g_l x^{k_l} g_{l-1} x^{k_{l-1}} \dots g_1 x^{k_1} g_0 x^{k_0}.
    $$
    Further, we write $W^1_\calG$ for the set of all $w \in W^*_\calG$ with exactly one occurrence of $x$ or $x^{-1}$ and write $W^{>1}_\calG := W^*_\calG \setminus W^1_\calG$.
\end{defi}

\begin{rem}
    The definition of nice words given above is slightly stronger than in \cite{VFLS} as we require the last block of $x$ to be positive.
    However, as $w$ and $w^{-1}$ have the same fix-points for every word $w$, without loss of generality we may work with this slightly more restrictive notion of nice words. 
\end{rem}

\begin{defi}\label{DEF_Zhang_Forcing}
    Let $\calG$ be a cofinitary group.
    Then we define {\emph{Zhang's forcing}} $\bbZ_\calG$ to be the set of all pairs $(s,E)$ such that $s$ is a finite partial function $s:\omega \overset{\text{partial}}{\longrightarrow} \omega$ and $E \in [W^*_\calG]^{<\omega}$.
    We order $(t, F) \extends (s, E)$ by
    \begin{enumerate}
        \item $s \subseteq t$ and $E \subseteq F$,
        \item for all $w \in E$ we have $\fix(w[t]) = \fix(w[s])$.
    \end{enumerate}
\end{defi}

\begin{rem}\label{REM_Zhang_Forcing}
    For any cofinitary group $\calG$ the forcing $\bbZ_\calG$ is $\sigma$-centered and if $G$ is $\bbZ_\calG$-generic in $V[G]$ we have that
    $$
        f_G= f_\text{gen} := \bigcup \set{s}{\Exists E \subseteq W^*_\calG \ (s,E) \in G } \in S_\infty.
    $$
    Further, $\calG \cup \simpleset{f_\text{gen}}$ generates a cofinitary group (e.g.\ see \cite{YZhang}).
\end{rem}

To this end, Zhang's forcing satisfies the following well-known domain and range extension lemma, which forces the generic function to be a permutation (e.g.\ see \cite{YZhang} or \cite{VFLS}).

\begin{lem}\label{LEM_Dom_Ran_Extension}
    Let $\calG$ be a cofinitary group and $(s,E) \in \bbZ_\calG$.
    Then we have
    \begin{enumerate}
        \item if $n \notin \dom(s)$, then for almost all $m < \omega$ we have $(s \cup \simpleset{(n,m)}, E) \extends (s,E)$,
        \item if $m \notin \ran(s)$, then for almost all $n < \omega$ we have $(s \cup \simpleset{(n,m)}, E) \extends (s,E)$.
    \end{enumerate}
\end{lem}

Further, to show that Zhang's forcing yields tightness in the sense of Definition~\ref{DEF_Tight_Cofinitary} we need the following additional technical lemmas, which yield a stronger version of the standard generic hitting lemma required for tightness.

\begin{lem}\label{LEM_Many_Extensions}
    Let $\calG$ be a cofinitary group.
    Let $T \in \calI_i(\calG)^+$, $t \in T$, $(s, E_0) \in \bbZ_\calG$ and $N < \omega$ with $E_0 \subseteq W^1_\calG$.
    Then, there is a $t' \in T$ with $t \subseteq t'$ satisfying
    $$
        \left| \set{k \in \dom(t') \setminus \dom(t)}{(s \cup \simpleset{(k, t'(k))}, E_0) \extends (s, E_0)} \right| > N.
    $$
\end{lem}

\begin{proof}
    It suffices to verify the case $N = 0$.
    The general case then follows inductively.
    Since $T \in \calI_i(\calG)^+$, choose $t' \in T$ with $t \subseteq t'$ and $k \in \dom(t') \setminus \dom(t)$ with $k \notin \dom(s)$, $t'(k) \notin \ran(s)$ and $t'(k) \neq g(k)$ for all $g \in E_0 \restr \calG$.
    But then, for $s' := s \cup \simpleset{(k, t'(k))}$ we claim that $(s', E_0) \extends (s, E_0)$.
    But by choice of $k$ we have that $s'$ is still a partial injection.
    Further, for every $w \in E_0$ we may choose $g \in \calG$ such that $w = gx$.
    By choice of $k$ we have $t'(k) \neq g^{-1}(k)$.
    Thus, we compute
    $$
        w[s'](k) = g(s'(k)) = g(t'(k)) \neq k.
    $$
    Finally, for every $l \in \ran(s) \cup \dom(s)$ we have $w[s](l) = w[s'](l)$, so that $\fix(w[s]) = \fix(w[s'])$.
\end{proof}

\begin{lem} \label{LEM_Tree_Extension}
    Let $\calG$ be a cofinitary group.
    Let $T \in \calI_i(\calG)^+$, $t \in T$ and $(s, E) \in \bbZ_\calG$.
    Then, there is $(s', E) \in \bbZ_\calG$, $t' \in T$ and $k \in \dom(t') \setminus \dom(t)$ such that $(s', E) \extends (s, E)$, $t \subseteq t'$ and $t'(k) = s'(k)$.
\end{lem}

\begin{proof}
    Let $E = E_0 \cup E_1$, where $E_0 \subseteq W^1_{\calG}$ and $E_1 \subseteq W^{>1}_{\calG}$.
    We may assume that for all $w \in E_1$ every subword of $w$ in $W^1_{\calG}$ is in $E_0$.
    Choose $N < \omega$ large enough such that
    \begin{enumerate}[(N1)]
        \item $\dom(s) \cup \ran(s) \cup \dom(t) \subseteq N$,
        \item for all $g \in E \restr \calG$ we have $g[\dom(s) \cup \ran(s)] \subseteq N$,
        \item for all $g \in E \restr \calG \setminus \simpleset{\id}$ we have $\fix(g) \subseteq N$.
    \end{enumerate}
    By Lemma~\ref{LEM_Many_Extensions} choose $t' \in T$ with $t \subseteq t'$ satisfying
    $$
        \left| \set{k \in \dom(t') \setminus \dom(t)}{(s \cup \simpleset{(k, t'(k))}, E_0) \extends (s, E_0)} \right| > N.
    $$
    By injectivity of $t'$ choose such a $k_0 < \omega$ with $t'(k_0) \geq N$.
    We define $s' := s \cup \simpleset{(k_0, t'(k_0))}$ and claim that $(s', E)$ is as desired.
    First, since $N \leq k_0, s'(k_0)$ and by (N1) we have that $s'$ is still a partial injection.
    Thus, by choice of $k_0$ it suffices to verify that for every $w \in E_1$ we have that $\fix(w[s]) = \fix(w[s'])$, so let $w \in E_1$.
    As $w$ is nice and has at least two occurrences of $x$ or $x^{-1}$, we may write $w = vxux$ or $w = vx^{-1}ux$ for some $u \in \calG$ and $v \in W_\calG$.
    First, notice that for any $k \in \dom(s) \cup \ran(s)$ we have that
    $$
        w[s](k) = w[s'](k) \qquad (\text{in particular also if both sides are undefined}),
    $$
    as by (N1) and (N2) the computation along $w[s']$ starting with $k$ always stays below $N$.
    Thus, we may finish the proof by showing that $w[s'](k_0)$ is undefined.
    In case that $w = vxux$, we have $ux \in E_0$.
    Thus, by $(s', E_0) \extends (s, E_0)$ we have $(ux)[s'](k_0) \neq k_0$.
    But by (N2) we also have that $(ux)[s'](k_0) \notin \dom(s)$.
    Hence, $w[s'](k_0)$ is undefined.
    
    Otherwise, $w = vx^{-1}ux$, so that $u \neq \id$.
    Then, by (M3) we get $(ux)[s'](k_0) = u(t'(k_0)) \neq t'(k_0)$.
    But by (N2) we also have that $(ux)[s'](k_0) \notin \ran(s)$.
    Hence, $w[s'](k_0)$ is undefined.
\end{proof}

\begin{thm} \label{THM_Tight_Cofinitary_Groups_Exist}
    Assume ${\sf MA}(\sigma\text{-centered})$.
    Then, every cofinitary group of size $<\!\frc$ is contained in a tight cofinitary group of size $\frc$.
\end{thm}

\begin{proof}
    As in Theorem~2.4 in \cite{VFCS} for any cofinitary group $\calG$ of size $<\!\!\frc$ we use ${\sf MA}(\sigma\text{-centered})$ to diagonalize against every $\omega$-sequence of injective trees as in the definition of tightness in an iteration of length $\frc$.
    However, to this end we will instead use Zhang's forcing $\bbZ_\calG$ to obtain a cofinitary group, so we need to verify that the following sets are dense in $\bbZ_\calG$:
    \begin{enumerate}
        \item for every $n < \omega$ the set of all $(s,E) \in \bbZ_\calG$ such that $n \in \dom(s)$,
        \item for every $m < \omega$ the set of all $(s,E) \in \bbZ_\calG$ such that $m \in \ran(s)$,
        \item for every $T \in \calI_i(\calG)^+$, $t \in T$ the set of all $(s, E) \in \bbZ_\calG$ such that there is a $t' \in T$ with $t \subseteq t'$ and a $k \in \dom(t') \setminus \dom(t)$ with $s(k) = t'(k)$,
        \item for every $w \in W^*_\calG$ the set of all $(s,E) \in \bbZ_\calG$ such that $w \in E$.
    \end{enumerate}
    But (1) and (2) are dense by the domain and range extension Lemma~\ref{LEM_Dom_Ran_Extension}, (3) is dense by the previous Lemma~\ref{LEM_Tree_Extension} and the density of (4) is trivial.
    Hence, a generic $f$ hitting these $<\!\!\frc$-many dense sets will extend $\calG$ to a cofinitary group $\langle \calG \cup \simpleset{f} \rangle$ and diagonalize against a given witness of tightness.
\end{proof}

\section{Co-analyticity and Zhang's forcing with coding into orbits} \label{SEC_co-analytic}

In order to obtain a co-analytic tight cofinitary group, we present a new parameter-free coding technique for maximal cofinitary groups.
To this end, we will code a real using the parity of the length of the orbits of elements of our cofinitary group.
First, we present a modification of Zhang's forcing, which codes a real into the lengths of orbits of the Zhang generic real.
This will yield a tight cofinitary group with a co-analytic set of generators.
Secondly, we will expand our coding technique, so that the orbit function of every new word codes some real.
Hence, we obtain that the entire tight cofinitary group is co-analytic.

\subsection{Coding into orbits of the Zhang generic real}

For any partial injection $s:\omega \overset{\text{partial}}{\longrightarrow} \omega$ we may visualize the function values of $s$ using its orbits.
As $s$ is partial its orbits may be closed or open, as visualized below:

\begin{center}
    \begin{tikzpicture}[->,>=stealth',auto,node distance=1.3cm, thick,main node/.style={rectangle,draw}]

    \node[main node] (1) {$4$};
    \node[main node] (2) [right of=1] {$6$};
    \node[main node] (3) [right of=2] {$42$};
    \node[main node] (4) [right of=3] {$1$};
    \node (9) [right of=4]{};
    \node[main node] (5) [right of=9] {$2$};
    \node[main node] (6) [right of=5] {$0$};
    \node[main node] (7) [right of=6] {$10$};
    \node[main node] (8) [right of=7] {$3$};

    \path[every node/.style={font=\sffamily\small}]
        (1) edge node [right] {} (2)
        (2) edge node [right] {} (3)
        (3) edge node [right] {} (4)
        (5) edge node [right] {} (6)
        (6) edge node [right] {} (7)
        (7) edge node [right] {} (8)
        (8) edge[bend right] node [left] {} (5);
    \end{tikzpicture}
    \end{center}

Here, on the left the set $\{1,4,6,42\}$ is an open orbit of $s$, i.e.\ $4 \notin \ran(s)$ and $1 \notin \dom(s)$ - on the right the set $O = \{ 0,2,3,10 \}$ is a closed orbit of $s$, i.e.\ $O \subseteq \dom(s) \cap \ran(s)$.
Note, that for full bijections from $\omega$ to $\omega$ all of their orbits look like $\bbZ$-chains or finite cycles.
We will see that for the generic $\dot{f}_\text{gen}$ added by $\bbZ_\calG$ in fact only finite orbits occur by generic closing of orbits.

\begin{defi}\label{DEF_Full_Orbit_Function}
    Given $f \in \omega^\omega$ and $n < \omega$ let $O_f(n)$ be the orbit of $f$ containing $n$, that is the smallest set containing $n$ closed under applications of $f$ and $f^{-1}$, and define $\calO_f := \set{O_f(n)}{n < \omega}$.
    There is a natural well-order on $\calO_f$ defined for $O, P \in \calO_f$ by $O < P$ iff $\min(O) < \min(P)$.
    Assume $f$ only has finite orbits; it follows that $f$ has infinitely many orbits.
    Then, we may define a function $o_f:\omega \to 2$ by
    $$
        o_f(n) := (\left|O_n\right| \text{ mod } 2),
    $$
    where $O_n$ is the $n$-th element in the well-order of $\calO_f$.
\end{defi}

\begin{defi}\label{DEF_Partial_Orbit_Function}
    For any finite partial injection $s: \omega \overset{\text{finite}}{\longrightarrow} \omega$ and $n < \omega$ define $O_s(n)$ and $\calO_s$ as above.
    We say {\emph{an orbit $O \in \calO_s$ is closed}} iff $O \subseteq \dom(s) \cap \ran(s)$ and denote with $\calO^c_s$ the set of all closed orbits of $s$.
    Conversely, we denote with $\calO^o_s := \calO_s \setminus \calO^c_s$ the set of all {\emph{open orbits}} of $s$.
    We say {\emph{$s$ is nice}} iff for all $O \in \calO^c_s$ we have $\min(O) < \min (\omega \setminus \bigcup \calO^c_s)$.
    For any nice $s$ define a function $o_s:\left|\calO_s\right| \to 2$ by
    $$
        o_s(n) := (\left|O_n\right| \text{ mod } 2),
    $$
    where $O_n$ is the $n$-th element in the well-order of $\calO^c_s$.
    For a real $r \in 2^\omega$ and a nice $s$ we say $s$ codes $r$ iff $o_s \subseteq r$.
\end{defi}

Note, that niceness makes sure that we do not prematurely close any orbit, in the sense that we do not know which function value should be coded, as the well-order of orbits is not decided up to that point yet.

\begin{defi}\label{DEF_Zhang_Coding_Forcing}
    Let $r \in 2^\omega$ be a real and $\calG$ be a cofinitary group.
    Then, we define $\bbZ_\calG(r)$ as the set of all elements $(s, E) \in \bbZ_\calG$ such that $s$ is nice and codes $r$, ordered by the restriction of the order on $\bbZ_\calG$.
\end{defi}

\begin{rem}\label{REM_Generic_Codes_Real}
    We will show that $\dot{f}_{\text{gen}}$ generically codes as much information of $r$ as desired.
    Hence, by definition of $\bbZ_\calG(r)$ in the generic extension we have $o_{\dot{f}_{\text{gen}}} = r$, and thus $r$ can be decoded from the generic Zhang real.
    First, we verify range and domain extension:
\end{rem}

\begin{lem}\label{LEM_Dom_Ran_Extension_Coding}
    Let $\calG$ be a cofinitary group and $(s,E) \in \bbZ_\calG(r)$.
    Then we have
    \begin{enumerate}
        \item if $n \notin \dom(s)$, then for almost all $m < \omega$ we have $(s \cup \simpleset{(n,m)}, E) \extends (s,E)$ as well as $(s \cup \simpleset{(n,m)}, E) \in \bbZ_\calG(r)$,
        \item if $m \notin \ran(s)$, then for almost all $n < \omega$ we have $(s \cup \simpleset{(n,m)}, E) \extends (s,E)$ as well as $(s \cup \simpleset{(n,m)}, E) \in \bbZ_\calG(r)$.
    \end{enumerate}
\end{lem}

\begin{proof}
    This follows immediately from Lemma~\ref{LEM_Dom_Ran_Extension}.
    Note, that possibly only one choice of $n$ (or $m$) may close an orbit of $s$, which immediately implies that $(s \cup \simpleset{(n,m)}, E) \in \bbZ_\calG(r)$ for almost all $n$ (or $m$).
\end{proof}

Secondly, we again verify the stronger generic hitting lemma required for tightness (as in Lemma~\ref{LEM_Tree_Extension}).
Hence, also $\bbZ_\calG(r)$ may be used to construct or force a tight cofinitary group.

\begin{lem}\label{LEM_Tree_Extension_Coding}
    Let $\calG$ be a cofinitary group.
    Let $T \in \calI_i(\calG)^+$, $t \in T$ and $(s, E) \in \bbZ_\calG(r)$.
    Then, there is $(s', E) \in \bbZ_\calG(r)$, $t' \in T$ and $k \in \dom(t') \setminus \dom(t)$ such that $(s', E) \extends (s, E)$, $t \subseteq t'$ and $t'(k) = s'(k)$.
\end{lem}

\begin{proof}
    The function pair $(n,m)$ added to $s$ in Lemma~\ref{LEM_Many_Extensions} satisfies $n, m \notin \ran(s) \cup \dom(s)$ and $n \neq m$.
    Hence, the extension $s \cup \simpleset{(n,m)}$ does not close an orbit of $s$, which immediately implies that $s \cup \simpleset{(n,m)} \in \bbZ_\calG(r)$.
\end{proof}

Finally, we need to verify that we may generically code as much information of $r$ as desired.
To this end, we prove that for any $(s,E) \in \bbZ_\calG$ and $O \in \calO^o_s$ and almost all $k < \om$ we may extend $(s,E)$ to close $O$ in length $k$.

\begin{lem}\label{LEM_Generic_Closing_Of_Orbits}
    Let $\calG$ be a cofinitary group, $(s, E) \in \bbZ_\calG$ and $n \in \omega \setminus \bigcup \calO^c_s$.
    Then, there is a $K < \omega$ such that for all $k > K$ there is $(t,E) \in \bbZ_\calG$ with $(t,E) \extends (s,E)$, $O_t(n) \in \calO^c_t$ and $\left|O_t(n)\right| = k$.
\end{lem}

\begin{proof}
    By Lemma~\ref{LEM_Dom_Ran_Extension} we may assume that $n \in \dom(s) \cup \ran(s)$.
    As $O_s(n) \in \calO^o_s$ choose $n_{-} < \omega$ to be the unique element of $O_s(n) \setminus \ran(s)$ and $n_+$ be the unique element of $O_s(n) \setminus \dom(s)$.
    Let $L := \max \set{\left|w\right|}{w \in E}$ and set $K := \left|O_s(n)\right| + L$.
    Now, let $k > K$.
    Choose $L' \geq L$ and pairwise different natural numbers $A := \simpleset{a_0, \dots, a_{L' - 1}}$ such that
    \begin{enumerate}
        \item $\left|O_s(n)\right| + L' = k$,
        \item for all $a \in A$ we have
        \begin{enumerate}
            \item $a \notin \dom(s) \cup \ran(s)$,
            \item for all $g \in E \restr \calG \setminus \simpleset{\id}$ we have $g(a) \notin \dom(s) \cup \ran(s) \cup A$.
        \end{enumerate}
    \end{enumerate}
    Note, that we can ensure (2b) as every $g \in E \restr \calG \setminus \simpleset{\id}$ only has finitely many fixpoints.
    We claim that for
    $$
        t := s \cup \simpleset{(n_{+}, a_0)} \cup \set{(a_i, a_{i + 1})}{i < L' - 1}  \cup \simpleset{(a_{L' - 1}, n_{-})}
    $$
    we have that $(t, E) \in \bbZ_\calG$.
    Clearly, then $(t,E) \extends (s,E)$, $O_t(n) \in \calO^c_t$ and $\left|O_t(n)\right| = k$ by (1).
    Visualized, the orbit $O_t(n)$ then looks as follows:

    
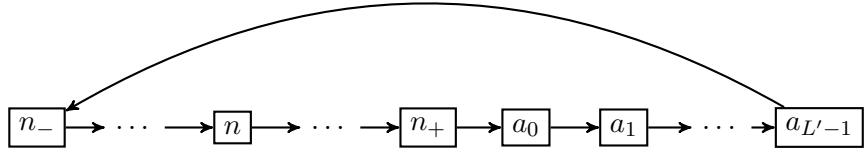
\begin{center}
    \begin{tikzpicture}[->,>=stealth',auto,node distance=1.3cm, thick,main node/.style={rectangle,draw}]

    \node[main node] (1) {$n_{-}$};
    \node (2) [right of=1] {$\dots$};
    \node[main node] (3) [right of=2] {$n$};
    \node (4) [right of=3] {$\dots$};
    \node[main node] (5) [right of=4] {$n_{+}$};
    \node[main node] (6) [right of=5] {$a_0$};
    \node[main node] (7) [right of=6] {$a_1$};
    \node (8) [right of=7] {$\dots$};
    \node[main node] (9) [right of=8] {$a_{L'-1}$};

    \path[every node/.style={font=\sffamily\small}]
        (1) edge node [right] {} (2)
        (2) edge node [right] {} (3)
        (3) edge node [right] {} (4)
        (4) edge node [right] {} (5)
        (5) edge node [right] {} (6)
        (6) edge node [right] {} (7)
        (7) edge node [right] {} (8)
        (8) edge node [right] {} (9)
        (9) edge[bend right] node [left] {} (1);
    \end{tikzpicture}
    \end{center}
    So, let $w \in E$.
    If $w = x^{k_0}$ for some $k_0 > 0$, note that $\fix(w[s]) = \fix(w[t])$ as by choice of $L$ we have $\left|O_t(n)\right| = k > L \geq \left|w\right| = k_0$.
    Hence, we may write $w = gv$ for some $g \in \calG \setminus \simpleset{\id}$ and $v \in W_\calG$.
    Towards  contradiction, assume that $d \in \fix(w[t]) \setminus \fix(w[s])$.
    First, note that $v[t](d) \in \dom(s) \cup \ran(s) \cup A$.
    Thus, by (2b) we have $d = w[t](d) = (gv)[t](d) \notin A$.
    Hence, $d \in \dom(s) \cup \ran(s)$.

    If the entire computation of $d$ along $w[t]$ stays in $\dom(s) \cup \ran(s)$, then $w[t](d) = w[s](d)$, contradicting $d \notin \fix(w[s])$.
    Thus, write $w = w_1 w_0$ with $w_0, w_1 \in W_\calG$ with $w_0$ minimal such that $w_0[t](d) \notin \dom(s) \cup \ran(s)$, i.e.\ $w_0[t](d) \in A$.
    By (2b) the left-most letter of $w_0$ has to be $x$ or $x^{-1}$, w.l.o.g.\ assume $w = w_1 x w_0'$ (the other case is symmetric).
    As we have $(w_0')[t](d) \in \dom(s) \cup \ran(s)$ and $(xw_0')[t](d) \in A$ we get $(w_0')[t](d) = n_{+}$.
    Finally, write $w = w_1'gx^lw_0'$ for some $w_1' \in W_\calG$, $g \in \calG \setminus \simpleset{\id}$ and $l > 0$.
    Then, $l < \left|w\right| \leq L \leq L'$ implies that
    $$
        (x^lw_0')[t](d) = (x^l)[t](n_{+}) \in A.
    $$
    Hence, by (2b) we have $(gx^lw_0')[t](d) \notin \dom(s) \cup \ran(s) \cup A$.
    Thus, if $w_1'$ is the empty word, this contradicts $d \in \dom(s) \cup \ran(s)$ and otherwise $(w_1'gx^lw_0')[t](d)$ is undefined, a contradiction.
\end{proof}

\begin{cor}\label{COR_Generic_Coding}
    Let $\calG$ be a cofinitary group and $(s,E) \in \bbZ_\calG(r)$.
    Then, there is $(t, E) \in \bbZ_\calG(r)$ such that $(t, E) \extends (s, E)$ and $\left|\calO^c_s\right| < \left|\calO^c_t\right|$.
\end{cor}

\begin{proof}
    Let $n := \min (\omega \setminus \bigcup \calO^c_s)$.
    By Lemma~\ref{LEM_Generic_Closing_Of_Orbits} there is $(t, E) \in \bbZ_\calG$ such that $(t, E) \extends (s,E)$, $O_t(n) \in \calO^c_t$ and
    $$
        \left|O_t(n)\right| \equiv r(\left|\calO^c_s\right|) \qquad (\text{mod } 2).
    $$
    Further, $t$ can be chosen to not close any other orbits, so that by choice of $n$ we have that $t$ is nice and codes $r$.
    Hence, $(t,E) \in \bbZ_\calG(r)$.
\end{proof}

\begin{cor}\label{COR_Orbit_Function_Unbounded}
    Let $\calG$ be a cofinitary group.
    Then we have
    $$
        \bbZ_\calG \forces \dot{f}_{\text{gen}} \text{ only has finite orbits and } o_{\dot{f}_{\text{gen}}} \text{ is an unbounded real over } V.
    $$
\end{cor}

\begin{proof}
    Immediately follows from Lemma~\ref{LEM_Generic_Closing_Of_Orbits} as generically we may close any open orbit in arbitrarily long length.
\end{proof}

\subsection{Coding into orbits of every new word}

In the last section we have seen how to code a real into the orbit function of the Zhang generic.
Now, we will consider a slightly different orbit function, which is stable under cyclic permutations and inverses thereof.
Hence, in order to code a real into every new element of our cofinitary group, we will be able to restrict to nice words.

\begin{defi} \label{DEF_Stable_Orbit_Function}
    Assume $f \in \omega^\omega$ only has finite orbits and for every $n < \omega$ only finitely many orbits of length $n$.
    Let $\set{p_n}{n < \omega}$ enumerate all primes.
    Then, define a function $o_f^\dagger:\omega \to 2$ by
    $$
        o^\dagger_f(n) := (\left|\set{O \in \mathcal{O}_f}{\left|O\right| = p_n}\right| \text{ mod } 2).
    $$
    Similarly, for any finite partial injection $s: \omega \overset{\text{finite}}{\longrightarrow} \omega$ define a function $o^\dagger_s:\omega \to 2$ by
    $$
        o^\dagger_s(n) := (\left|\set{O \in \mathcal{O}^c_s}{\left|O\right| = p_n}\right| \text{ mod } 2).
    $$
    Hence, for every $n < \omega$ we are counting how many orbits of size $p_n$ there are.
    For a real $r$ and $n < \om$ we say $s$ codes $r$ up to $n$ iff $r \restr (n + 1) = o^\dagger_s \restr (n + 1)$.
    Finally, we say $f$ codes $r$ iff $o^\dagger_f = r$.
\end{defi}

\begin{rem} \label{LEM_Stable_Orbit_Function_Well_Defined}
    In the previous section, we have seen that the Zhang generic only has closed orbits (see Lemma~\ref{LEM_Generic_Closing_Of_Orbits}).
    Further, note that for every $n < \omega$ closing a new orbit in length $n$ implies that $x^n$ has a new fixpoint.
    Hence, every $(s,E) \in \mathbb{Z}_\calG$ with $x^n \in E$ forces that the number of orbits of length of $\dot{f}_{\text{gen}}$ is decided, i.e.\ $\dot{f}_{\text{gen}}$ only has finitely many orbits of length $n$.
    Hence, $\dot{f}_{\text{gen}}$ satisfies the assumption of Definition~\ref{DEF_Stable_Orbit_Function}.
\end{rem}

\begin{rem} \label{REM_Orbits_Under_Inverse}
    As $f$  and $f^{-1}$ have the same orbits, we obtain that if $f$ only has finite orbits and for every $n < \omega$ only finitely many orbits of length $n$, then the same holds for $f^{-1}$ and $o^\dagger_f = o^\dagger_{f^{-1}}$.
\end{rem}

\begin{lem} \label{LEM_Orbits_Under_Cyclic_Permutations}
    Let $f, g \in \omega^\omega$ be bijections.
    Then, the map $\pi:\mathcal{O}_{fg} \to \mathcal{O}_{gf}$ defined for $O \in \mathcal{O}_{fg}$ by
    $$
        \pi(O) = g[O],
    $$
    defines a bijection.
    Further, $\pi$ maps every orbit of $fg$ to an orbit of $gf$ of the same length.
\end{lem}

\begin{proof}
    The second part follows immediately as $g$ is a bijection.
    First, we verify that $\pi$ maps to $\mathcal{O}_{gf}$.
    So, let $O \in \mathcal{O}_{fg}$, i.e.\ $O$ is a minimal non-empty set closed under applications of $fg$ and $(fg)^{-1}$.
    Clearly, $\pi[O]$ is non-empty and for $m \in g[O]$ choose $n < \om$ such that $g(n) = m$.
    Then, we compute
    $$
        (gf)(m) = (gfg)(n) = g(fg(n)) \in g[O],
    $$
    as $fg(n) \in O$.
    Similarly, we have
    $$
        (gf)^{-1}(m) = (f^{-1}g^{-1}g)(n) = (gg^{-1}f^{-1})(n) = g((fg)^{-1}(n)) \in g[O],
    $$
    as $(fg)^{-1}(n) \in O$.
    Hence, $g[O]$ is closed under applications of $gf$ and $(gf)^{-1}$.
    Now, let $P \subseteq g[O]$ be a non-empty subset closed under applications of $gf$ and $(gf)^{-1}$.
    By the same argument as above we have that $g^{-1}[P] \subseteq O$ is a non-empty set closed under applications of $fg$ and $(fg)^{-1}$.
    But $O \in \mathcal{O}_{fg}$, so that $g^{-1}[P] = O$.
    Hence, $P = g[O]$, which shows that $g[O] \in \mathcal{O}_{gf}$.
    Finally, note that by the same argumentation $\psi:\mathcal{O}_{gf} \to \mathcal{O}_{fg}$ defined for $O \in \mathcal{O}_{fg}$ by
    $$
        \psi(O) = g^{-1}[O],
    $$
    is well-defined and clearly the inverse of $\pi$.
    Hence, $\pi$ is bijective.
\end{proof}

\begin{cor} \label{COR_Orbit_Function_Under_Cyclic_Permuations}
    Let $f,g \in \omega^\omega$ be bijections. Assume $fg$ only has finite orbits and for every $n < \omega$ only finitely many orbits of length $n$.
    Then $gf$ has only finite orbits, for every $n < \omega$ only finitely many orbits of length $n$ and $o^\dagger_{fg} = o^\dagger_{gf}$.
\end{cor}

Next, we will show that our orbit function is also stable under finite powers, so that we may restrict to even nicer words, in whose orbit functions we will code a real in the end.
This is the place, where we use the sequence of primes in Definition~\ref{DEF_Stable_Orbit_Function}.

\begin{defi}\label{DEF_Very_Nice_Words}
    Let $\calG$ be a cofinitary group.
    We say a nice word $w \in W^*_\calG$ is {\emph{indecomposable}} iff there is no $v \in W^*_\calG$ and $k > 1$ such that $v^k = v \underset{k\text{-times}}{\dots}v = w$.
    We denote with $W^\dagger_\calG$ the set of all indecomposable words.
\end{defi}

\begin{lem}\label{LEM_Orbit_Function_Of_Powers}
    Let $f \in \omega^\omega$ be a bijection with only finite orbits and for every $n < \omega$ only finitely many orbits of length $n$ and let $k < \omega$.
    Then, $f^k$ only has finite orbits, for every $n < \omega$ only finitely many orbits of length $n$ and $o^\dagger_f(n) = o^\dagger_{f^k}(n)$ for almost all $n < \omega$.
\end{lem}

\begin{proof}
    Every orbit of $f^k$ of size $n$ is contained in an orbit of $f$ of length at most $kn$, so that $f^k$ only has finite orbits and for every $n < \omega$ only finitely many orbits of length $n < \omega$.

    In particular, $f^k$ only has finitely many orbits of length $1$, i.e.\ only finitely many fixpoints.
    We show for almost all $n < \omega$ that $o^\dagger_f(n) = o^\dagger_{f^k}(n)$.
    So assume $p_n > \left|\fix(f^k)\right|$, where $p_n$ is the $n$-th prime number.
    Fix an orbit $O \in \mathcal{O}_f$ of size $p_n$ and let $\mathcal{P}$ be the set of all $P \in \mathcal{O}_{f^k}$ with $P \subseteq O$.
    Applying $f$ induces bijections between members of $\mathcal{P}$, i.e.\ all of them have the same size.
    As $\mathcal{P}$ partitions $O$ we obtain that $\left|P\right|$ divides $\left|O\right| = p_n$ for all $P \in \mathcal{P}$.
    But $p_n$ is prime, so either $\left|P\right| = 1$ for all $P \in \mathcal{P}$, which implies that $f^k$ has at least $\left|O\right| = p_n$-many fixpoints, contradicting the choice of $n$.
    Thus, $\left|\mathcal{P}\right| = 1$, i.e.\ $f$ and $f^k$ have the same number of orbits of size $p_n$.
    Hence, we proved $o^\dagger_f(n) = o^\dagger_{f^k}(n)$.
\end{proof}

\begin{cor}\label{COR_Restrict_To_Orbits_Of_Indecomposable_Words}
    Let $\calG$ be a cofinitary group and let $f \in \omega^\omega \setminus \calG$ such that $\simpleseq{\calG \cup \simpleset{f} }$ is cofinitary.
    Further, assume that there is $r \in 2^\omega$ such that for every $w \in W^\dagger_\calG$ we have $w[f]$ codes $r$.
    Then, for every $g \in \simpleseq{\calG \cup \simpleset{f} } \setminus \calG$ we have $g$ almost codes $r$, i.e.\ $r(n) = o^\dagger_g(n)$ for almost all $n < \omega$.
\end{cor}

\begin{proof}
    Let $g \in \simpleseq{\calG \cup \simpleset{f} } \setminus \calG$.
    By our slight modification of nice words (see Definition~\ref{DEF_Nice_Words}) and the properties of nice words in \cite{VFLS}, we may choose $w_0, w_1 \in W_\calG$ such that $w:=w_0w_1 \in W^*_\calG$ and $g = (w_1w_0)[f]$ or $g^{-1} = (w_1w_0)[f]$.
    By Remark~\ref{REM_Orbits_Under_Inverse} and Lemma~\ref{LEM_Orbits_Under_Cyclic_Permutations} it suffices to verify that $r(n) = o^\dagger_{w[f]}(n)$ for almost all $n < \omega$.
    If $w \in W^\dagger_\calG$ we are done by assumption, so let $v \in W^\dagger_\calG$ and $k > 1$ such that $w = v^k$.
    Then, $v[f]$ codes $r$, so by Lemma~\ref{LEM_Orbit_Function_Of_Powers} we obtain
    $$
        o^\dagger_{w[f]}(n) = o^\dagger_{v^k[f]}(n) = o^\dagger_{v[f]}(n) = r(n)
    $$
    for almost all $n < \omega$.
\end{proof}

Hence, we will only have to make sure that every indecomposable word codes $r$.
Next, we introduce the variation of Zhang's forcing, which ensures exactly this property.

\begin{defi}\label{DEF_Zhang_Coding_In_Every_Word_Forcing}
    Let $r \in 2^\omega$ be a real and $\calG$ be a cofinitary group.
    Then, we define $\bbZ^\dagger_\calG(r)$ as the set of all elements $(s, E) \in \bbZ_\calG$ such that $E$ is closed under cyclic permutations and inverses thereof in $W^*_\calG$, and for all $w \in E$ if $w = v^k$ for $v \in W^\dagger_\calG$ and $k < \omega$, then $v^l \in E$ for all $0 < l \leq k$ and for all $n < \omega$ with $p_n \leq k$ also $o^\dagger_{v[s]}$ codes $r$ up to $n$.
    We let the order on $\bbZ^\dagger_\calG(r)$ be the restriction of the order on $\bbZ_\calG$.
\end{defi}

\begin{prop}\label{PROP_Simple_Extensions}
    Let $\calG$ be a cofinitary group and $(s, E) \in \bbZ^\dagger_\calG(r)$.
    Then, for every $(t, E) \in \bbZ_\calG$ with $(t, E) \extends (s,E)$ we have $(t, E) \in \bbZ^\dagger_\calG(r)$.
\end{prop}

\begin{proof}
    If not, there are $v \in W_\calG^\dagger$, $n < \omega$ and $O \in \mathcal{O}^c_t \setminus \mathcal{O}^c_s$ such that $v^{p_n} \in E$ and $\left|O\right| = p_n$.
    But then, for every $k \in O$ we have $k \in \fix(v^{p_n}[t]) \setminus \fix(v^{p_n}[s])$, contradicting $(t, E) \extends (s, E)$.
\end{proof}

\begin{cor}\label{COR_Dom_Ran_Extension_Stable_Coding}
    Let $\calG$ be a cofinitary group and $(s,E) \in \bbZ^\dagger_\calG(r)$.
    Then we have
    \begin{enumerate}
        \item if $n \notin \dom(s)$, then for almost all $m < \omega$ we have $(s \cup \simpleset{(n,m)}, E) \extends (s,E)$ as well as $(s \cup \simpleset{(n,m)}, E) \in \bbZ^\dagger_\calG(r)$,
        \item if $m \notin \ran(s)$, then for almost all $n < \omega$ we have $(s \cup \simpleset{(n,m)}, E) \extends (s,E)$ as well as $(s \cup \simpleset{(n,m)}, E) \in \bbZ^\dagger_\calG(r)$,
    \end{enumerate}
\end{cor}

\begin{cor} \label{LEM_Tree_Extension_Stable_Coding}
    Let $\calG$ be a cofinitary group.
    Let $T \in \calI_i(\calG)^+$, $t \in T$ and $(s, E) \in \bbZ^\dagger_\calG(r)$.
    Then, there is $(s', E) \in \bbZ^\dagger_\calG(r)$, $t' \in T$ and $k \in \dom(t') \setminus \dom(t)$ such that $(s', E) \extends (s, E)$, $t \subseteq t'$ and $t'(k) = s'(k)$.
\end{cor}

\begin{proof}
    Follows immediately from Proposition~\ref{PROP_Simple_Extensions} and Lemma~\ref{LEM_Dom_Ran_Extension} or Lemma~\ref{LEM_Tree_Extension}, respectively.
\end{proof}

Hence, we only have to verify, that densely we may add any nice word to the second component of every condition in $\bbZ^\dagger_\calG(r)$.
To this end, we verify that we can find extensions which add exactly one orbit to some indecomposable word.

\begin{lem} \label{LEM_Strong_Closure_Of_Orbits}
    Let $\calG$ be a cofinitary group.
    Let $(s, E) \in \bbZ^\dagger_\calG(r)$, $v \in W_\calG^\dagger$ and $k < \omega$ such that $v^k \notin E$.
    Then there is $(t, E) \in \bbZ^\dagger_\calG(r)$ such that $(t, E) \extends (s,E)$ and $\mathcal{O}^c_{v[t]} = \mathcal{O}^c_{v[s]} \cup \simpleset{O}$ for some $O \in \mathcal{O}^c_{v[t]} \setminus \mathcal{O}^c_{v[s]}$ with $\left|O\right| = k$.
\end{lem}

\begin{proof}
    Let $F := E \cup \simpleset{v}$.
    Choose $N < \omega$ such that
    \begin{enumerate}
        \item for all $g \in F \restr \calG$ we have $g[\dom(s) \cup \ran(s)] \subseteq N$,
        \item for all $g_0 \neq g_1 \in F \restr \calG$ and $n > N$ we have $g_0(n) \neq g_1(n)$.
    \end{enumerate}
    In particular, as $\id \in F \restr \calG$, we have
    \begin{enumerate}
        \item $\dom(s) \cup \ran(s) \subseteq N$,
        \item for all $g \in F \restr \calG \setminus \simpleset{\id}$ we have $\fix(g) \subseteq N$.
    \end{enumerate}
    Inductively, we define a sequence of pairwise different natural numbers $A := \set{a_i}{i < k \cdot\left|v\right|}$ and a sequence of finite partial injections $\seq{t_i}{0 < i \leq k \cdot\left|v\right|}$ as follows.
    Write $v^k = v_{k \cdot\left|v\right| - 1} \dots v_0$, choose $a_1 > N$ arbitrarily and set $t_1 := s$.
    Now, assume $a_i$ and $t_{i}$ are defined for some $0 < i < k \cdot\left|v\right| - 1$.
    If $v_i \in \calG$ set $a_{i + 1} := v_i(a_i)$ and $t_{i + 1} := t_{i}$.
    Otherwise, $v_i = x$ or $v_i = x^{-1}$.
    Choose, $a_{i + 1} > N$ such that for all $g \in F \restr \calG$ we have $g(a_{i + 1}) \notin\simpleset{a_1, \dots, a_{i}}$.
    Further, if $v_i = x$ set $t_{i + 1} := t_{i} \cup \simpleset{(a_i, a_{i + 1})}$ and if $v_i = x^{-1}$ set $t_{i + 1} := t_{i} \cup \simpleset{(a_{i + 1}, a_{i})}$.
    
    Finally, assume $a_{k \cdot\left|v\right| - 1}$ and $t_{k \cdot\left|v\right| - 1}$ are defined.
    In case $v_{k \cdot\left|v\right| - 1} \in \calG$, set $a_{0} :=  v_{k \cdot\left|v\right| - 1}(a_{k \cdot\left|v\right| - 1})$ and define $t_{k \cdot\left|v\right|} := t_{k \cdot\left|v\right| - 1} \cup \simpleset{(a_0, a_1)}$.
    Otherwise, $v_{k \cdot\left|v\right| - 1} = x$ (so that $v = x$, because $v \in W^\dagger_\calG$).
    Then, we may choose $a_{0} > N$ such that for all $g \in F \restr \calG$ we have $g(a_{0}) \notin\simpleset{a_1, \dots, a_{k \cdot\left|v\right| - 1}}$ and we set $t_{k \cdot\left|v\right|} := t_{k \cdot\left|v\right| - 1} \cup \simpleset{(a_{k \cdot\left|v\right| - 1}, a_0), (a_{0}, a_{1})}$.

    By construction, we have that all $a_i$ are pairwise distinct, $O := \set{a_{i \cdot \left|v\right|}}{i < k} \in \mathcal{O}^c_{v[t]} \setminus \mathcal{O}^c_{v[s]}$ and $\left|O\right| = k$.
    Let $t := t_{k \cdot \left|v\right|}$.
    Note, that for every $i < k \cdot\left|v\right|$ exactly one of the following cases is satisfied:
    \begin{enumerate}
        \item $a_i \in \dom(t) \cap \ran(t)$ and for every $g \in F \restr \calG \setminus \simpleset{\id}$ we have that $g(a_i) \notin \dom(t) \cup \ran(t)$ and $ \simpleset{t(a_i), t^{-1}(a_i)} = \simpleset{a_{i - 1}, a_{i + 1}}$,
        \item $a_i \in \dom(t) \setminus \ran(t)$ and there is a unique $g \in F \restr \calG \setminus \simpleset{\id}$ such that $g(a_i) \in \dom(t) \cup \ran(t)$ and $ \simpleset{t(a_i), g(a_i)} = \simpleset{a_{i - 1}, a_{i + 1}}$,
        \item $a_i \in \ran(t) \setminus \dom(t)$ and there is a unique $g \in F \restr \calG \setminus \simpleset{\id}$ such that $g(a_i) \in \dom(t) \cup \ran(t)$ and $ \simpleset{t^{-1}(a_i), g(a_i)} = \simpleset{a_{i - 1}, a_{i + 1}}$,
    \end{enumerate}
    where $a_{-1} := a_{k \cdot\left|v\right| - 1}$ and $a_{k \cdot\left|v\right|} := a_0$.
    Finally, it remains to show that $(t, E) \extends (s,E)$ and $\mathcal{O}^c_{v[t]} = \mathcal{O}^c_{v[s]} \cup \simpleset{O}$.
    
    So, let $w \in E$ and assume $d \in \fix(w[t]) \setminus \fix(w[s])$.
    By choice of $N$ we have $d \in A$ and the entire computation of $d$ along $v[t]$ is in $A$.
    So choose $i < k \cdot\left|v\right|$ with $d = a_i$ and let $w_0$ be the rightmost letter of $w$.
    By the three properties above and the fact that the entire computation of $d$ along $w[t]$ is defined, we obtain $w_0[t](a_i) = a_{i + 1}$ or $w_0[t](a_i) = a_{i - 1}$.
    By considering the word $w^{-1}$ we may restrict to the first case.
    Write $v = v_1 v_0$ for $v_0,v_1 \in W_\calG \setminus \simpleset{\id}$ such that $v_0(a_0) = a_i$.
    Inductively, using the three properties above, there is only one computation using $F \restr \calG \setminus \simpleset{\id}$ and $t, t^{-1}$ starting at $a_i$ and proceeding to $a_{i + 1}$.
    As $d$ is a fixpoint of $w[t]$, the length of $w$ is a multiple of $\left|A\right| = k \cdot\left|v\right|$.
    But then, by uniqueness of the computation we have $w = ((v_0v_1)^k)^l$ for some $l > 0$.
    But $v^k \notin E$, so by the closure of $E$ (remember Definition~\ref{DEF_Zhang_Coding_In_Every_Word_Forcing}) we get $w \notin E$, a contradiction.

    Finally, we have to show that we only added exactly one orbit to $\mathcal{O}^c_{v[s]}$.
    So let $a_i \in A \setminus O$ and assume $a_i \in P \in \mathcal{O}^c_{v[t]}$.
    Again, there are only two possible computations using $F \restr \calG \setminus \simpleset{\id}$ and $t, t^{-1}$ starting with $a_i$; one proceeding with $a_{i + 1}$ and the other with $a_{i-1}$.
    As before, by uniqueness we obtain the following two cases.
    Either, there are $v_0, v_1 \in W_\calG \setminus \simpleset{\id}$ with $v = v_1v_0$ and $v = v_0v_1$,
    or there are $v_0, v_1 \in W_\calG \setminus \simpleset{\id}$ with $v = v_1v_0$ and $v = v_0^{-1}v_1^{-1}$.
    It is easy to verify, that the first case contradicts that $v$ is indecomposable and the second case contradicts that $v$ is reduced, i.e. $v$ allows no cancellations.
\end{proof}

\begin{cor} \label{LEM_Adding_Words_Is_Dense}
    Let $\calG$ be a cofinitary group and $w \in W^*_\calG$.
    Then the set of all $(s, E)$ with $w \in E$ is dense in $\bbZ^\dagger_\calG(r)$.
\end{cor}

\begin{proof}
    Let $(s, E) \in \bbZ^\dagger_\calG(r)$ and $w = v^k$ for some $k < \omega$ and $v \in W^\dagger_\calG$ with $w \notin E$.
    Let $F$ be the closure of $E \cup \simpleset{w}$ under cyclic permutations and inverses thereof in $W^*_\calG$.
    If $k = 1$, then $(s, F) \in \bbZ^\dagger_\calG(r)$ and $(s,F) \extends (s, E)$.
    If $k > 1$ by induction we may assume $v^{k - 1} \in E$.
    If $k$ is not prime, we have $(s, F) \in \bbZ^\dagger_\calG(r)$ and $(s,F) \extends (s, E)$.
    So assume $k = p_n$ for some $n < \omega$.
    If
    $$
        o^\dagger_{v[s]}(n) \not\equiv r(n) \quad \mod 2,
    $$
    use the previous Lemma~\ref{LEM_Strong_Closure_Of_Orbits} to find $(t,E) \in \bbZ_\calG(r)$ with $(t, E) \extends (s,E)$ and
    $$
        o^\dagger_{v[t]}(n) \equiv r(n) \quad \mod 2.
    $$
    Hence, $v[t]$ codes $r$ up to $n$.
    By Proposition~\ref{REM_Orbits_Under_Inverse} and Corollary~\ref{COR_Orbit_Function_Under_Cyclic_Permuations} the same is true for every cyclic permutation of $v$ and inverses thereof in $W^\dagger_\calG$.
    Thus, $(t, F) \in \bbZ^\dagger_\calG(r)$, $(t, F) \extends (s, E)$ and $w \in F$.
\end{proof}

\begin{thm}\label{THM_Coanalytic_Group}
 In the constructible universe $L$ there is a co-analytic tight cofinitary group.
\end{thm}
\begin{proof}[Sketch of proof]
Equipped with the techniques developed in this section, we will inductively  construct a tight cofinitary group, following the general outline given in \cite[Theorem 4.1]{VFDSAT}.

We start by working in $L$. Construct a sequence $\langle \delta_\xi, z_\xi, \calG_\xi, \sigma_\xi   \mid \xi < \omega_1 \rangle$ satisfying the following:
\begin{enumerate}[(i)]
\item $\delta_\xi$ is a countable ordinal such that $L_{\delta_\xi}$ projects to $\omega$ and $\delta_\xi > \delta_\nu + \omega\cdot 2$ for each $\nu < \xi$,
\item $z_\xi \in 2^\omega$ such that the canonical surjection from $\omega$ to $L_{\delta_\xi}$ is computable from $ z_\xi$, 
\item $\calG_\xi$ is the group generated by $\{\sigma_\nu \colon \nu < \xi\}$,
\item $\sigma_\xi$ is the generic permutation over $L_{\delta_\xi}$ for Zhang's forcing with coding $z_\xi$ into orbits, over the group $\calG_\xi$.
\end{enumerate}
Such a sequence is easily constructed by induction: 
Supposing we already have constructed
$\langle \delta_\xi, z_\xi, \calG_\xi, \sigma_\xi  \mid \xi < \nu \rangle$
let 
$(\delta_\nu, z_\nu, \calG_\nu, \sigma_\nu )$ be the $\leq_L$-least triple such that the above items are satisfied.
Let
\[
\calG = \bigcup_{\xi<\omega_1} \calG_\xi.
\]
We have seen that $\calG$ is tight.
We verify that $\calG$ is co-analytic:
Fix a formula $\Psi(g)$ such that
\[
\Psi(g) \Leftrightarrow (\exists \nu < \omega_1)\; 
g = \langle \delta_\xi, z_\xi,  \calG_\xi, \sigma_\xi  \mid \xi < \nu \rangle
\]
and such that $\Phi$ is absolute for initial segments of $L$, i.e.,
\[
(\forall \alpha \in \text{ORD}) \; \big[ g \in L_\alpha \Rightarrow (\Psi(\vec g) \Leftrightarrow  L_\alpha\vDash \Psi(\vec g)) \big]
\]
A standard argument shows that 
\begin{multline*}
    g \in \calG \Leftrightarrow (\exists y \in 2^\omega) \text{ $y$ codes a well-founded model}\\
    \text{whose transitive collapse $M$ satisfies}\\
g \in M \land M \vDash (\exists \vec g) \; \Psi(\vec g) \land \vec g = \langle \delta_\xi, z_\xi, \sigma_\xi, \calG_\xi  \;\mid\; \xi < \nu + 1\rangle \land 
g \in \calG_\nu
\end{multline*}
where, crucially, the right-hand can be written $(\exists y \in 2^\omega)\; \Phi(y,g)$ with $\Phi$ a $\Pi^1_1$ formula.
We now show that
\begin{equation}\label{e.pi11}
g \in \calG \Leftrightarrow (\exists y \leq_h g)\; \Phi(y,g).
\end{equation}
It suffices to show $\Rightarrow$. But if $g \in \calG$, $g = w[\sigma_\xi]$ for some $\xi < \omega_1$, with $w$ a word in $\calG_\xi$, and so $z_\xi$ is computable in $g$. But $y$ as required is computable from $z_\xi$.
As in \cite[Theorem 4.1]{VFDSAT}, it can be shown that the right-hand side in \eqref{e.pi11} is equivalent to a $\Pi^1_1$ formula.
\end{proof}

\begin{rem}\label{RK1}
Note that, by \cite[Theorem 3.3]{VFCS} there are no analytic tight eventually different families, which implies that there are no analytic tight cofinitary groups and so the projective complexity of the group $\calG$ in the above theorem is optimal.
\end{rem}

\section{Indestructibility and strong preservation of tightness} \label{SEC_indestructibility}

In this section, we will show how to obtain a model with a co-analytic tight cofinitary group witnessing $\mathfrak{a}_g=\aleph_1$ in a model with a $\Delta^1_3$ well-order of the reals and $\mathfrak{c}=\aleph_2$. 

\subsection{Adjoining a $\Delta^1_3$ well-order via Sacks coding}  

In this subsection, we will spell out all main ingredients of the forcing poset, leading to \cite[Theorem 1]{VFSF}, which is the following statement:

\begin{thm*} There is a cardinal preserving generic extension of the constructible universe $L$ in which $\mathfrak{c}=\aleph_2$ and there is a $\Delta^1_3$ well-order of the reals. 
\end{thm*}

The extension is obtained via a countable support iteration of length $\aleph_2$ of $S$-proper forcing notions. Recall that for a stationary set $T\subseteq \omega_1$, a forcing notion $\mathbb{Q}$ is {\emph{$S$-proper}}, if for every countable elementary submodel $M$ of $H(\theta)$, where $\theta$ is sufficiently large and such that $M\cap \omega_1\in S$, every condition $p\in \mathbb{Q}\cap M$ has an $(\mathcal{M},\mathbb{Q})$-generic extension. Note that most preservation theorems for proper iterations easily transfer to preservation theorems for $S$-proper posets (see \cite{Goldstern}). Each iterand in the forcing construction of \cite[Theorem 1]{VFSF} is
\begin{itemize}
    \item either a poset, referred as {\emph{Localisation}} and denoted $\mathcal{L}(\varphi)$ (see \cite[Definition 1]{VFSF}; \item or is a poset designed to add a club disjoint from a given stationary, co-stationary set $S^*$, denoted $Q(S^*)$ and known as a poset for {\emph{shooting a club}} (see \cite[Definition 3]{VFSF}) 
    \item or a forcing notion known as {\emph{Sacks coding}} and denoted $C(Y)$ where $Y$ is some appropriate subset of $\omega_1$ (see \cite[Definition 2]{VFSF}). 
\end{itemize}    
The localization posets are proper and do not add reals (see \cite[Lemmas 3,4]{VFSF}); the posets of the form $Q(S^*)$ are  $\omega_1\backslash S^*$ proper (see \cite{Goldstern}) and $\omega$-distributive (see \cite{Jech}), and thus they also do not add reals. The only iterands adding reals in this iteration are the forcing notions of the form $C(Y)$.  For the purposes of our theorem \ref{main_two}, we will work with a slight strengthening of the original Sacks coding from \cite{VFSF} which can be found in \cite[Section 2.3]{JBVFCS}. The only difference between the poset from \cite{VFSF} and the forcing notion given below is item (1)(b). We proceed by giving the explicit definition of the Sacks coding posets:

\begin{defi}\label{DFN_Sacks_Coding}
Assume $V = L[Y]$ for a fixed $Y \subseteq \omega_1$ and that $\omega_1 = (\omega_1)^L$.
\begin{enumerate}
 \item  Inductively define a sequence  $\mu_i$ for $i < \omega_1$ as follows. Given $\langle \mu_j \mid j < i\rangle$, we define $\mu_i$ to be the least ordinal $\mu$ such that 
 \begin{enumerate}
     \item $\mu > \sup_{j<i} \mu_j$,
     \item  $L_\mu[Y\cap i]$ is $\Sigma^1_5$ elementary in $L_{\omega_1}[Y\cap i]$ (note that instead of $\Sigma^1_5$ we can take any sufficiently high finite level of the projective hierarchy; see also \cite[Remark 1]{JBVFCS}) and 
    \item  $L_\mu[Y\cap i]$  is a model of $\ZF^-$ with $L_\mu[Y\cap i]\vDash ``\omega\text{ is the largest cardinal}"$.
  \end{enumerate}    
\item We will say that  $r \in 2^\omega$ codes $Y$ below $i$, where $i < \omega_1$, if for all $j < i$ the following holds:
$$j \in Y\text{  if and only if }L_{\mu_j}[Y\cap j, r] \vDash \ZF^-.$$
\item For $i<\omega_1$, we denote by $\mathcal{A}_i$ the model $L_{\mu_i}[Y\cap i]$ and for a perfect tree $T\subseteq 2^{<\omega}$ we denote by $\lvert T\rvert$ the least $i$ such that $T\in \mathcal{A}_i$.
\end{enumerate}
{\emph{Sacks coding}}, denoted $C(Y)$, is the poset of all  perfect trees $T \subseteq 2^{<\omega}$ with the property that every branch of $T$ codes $Y$ below $\lvert T\rvert$. The extension relation is reverse inclusion.
\end{defi}

In the following, let $\mathbb{P}^*=\mathbb{P}_{\omega_2}$ be the countable support iteration of \cite[Theorem 1]{VFSF}, in which
every instance of a Sacks coding is substituted by the one given above. That is $\bbP^*$ is a countable support iteration of the form $\langle \bbP_\alpha,\dot{\bbQ}_\beta:\alpha\leq\omega_2, \beta<\omega_2\rangle$
where each iterand is $S$-proper and (forced to be) in one of the above three types. Then by \cite[Theorem 2]{VFSF} in $L^{\mathbb{P}^*}$ there is a $\Delta^1_3$-wellorder of the reals, while $\mathfrak{d}=\aleph_1<\mathfrak{c}=\aleph_2$. In what follows, we will show that $\mathbb{P}^*$ preserves the tightness of the co-analytic tight cofinitary group constructed in the previous section. 

\subsection{Strong preservation of tight cofinitary groups}

Key notions in analyzing the forcing preservation of tight mad families, as well as tight med families, are notions which appear natural strengthening of the notion of  $(M,\bbP)$-genericity (see  \cite[Definition 6.1]{GHT} and \cite[Definition 4.1]{VFCS} respectively). For cofinitary groups, the analogous notion is formulated below: 

\begin{defi}\label{DFN_Strongly_Preserves}
Let $\bbP$ be an $S$-proper forcing notion and let $\calG$ be a tight cofinitary group. 
We say that $\bbP$ {\emph{strongly preserves tightness of $\calG$}} if
for every sufficiently large regular cardinal $\theta$, every $p \in \bbP$ and every countable $M$ elementary submodel  $M$ of $H(\theta)$ such that $M\cap \omega_1\in S$, which contains $\bbP, p, \calG$ as elements the following holds: If $g\in\calG$ densely diagonalizes every element of  $M \cap \mathcal I_i(\calG)^+$, then there is an $(M, \bbP)$-generic condition $q$ extending $p$ and forcing that
$g$ densely diagonalizes every element in $M[\dot{G}_{\bbP}] \cap \mathcal I_i(\calG)^+$. We say that  
 $q$ is an {\emph{$(M,\bbP,\calG,g)$-generic condition}}. 
\end{defi}

The preservation theorems, developed in \cite{GHT}, giving the preservation of tight mad families under countable support iterations transfer to preservation theorems for tight eventually different families, see  \cite{VFCS}, 
as well as to preservation theorems needed for our purposes (see also \cite{JBVFCS}):

\begin{lem}(\cite[Lemma 6.3]{GHT}) Let $\calG$ be a tight cofinitary group. Let $\bbP$ be an $S$-proper forcing notion which strongly preserves the tightness of $\calG$ and let $\dot{\bbQ}$ be a $\bbP$-name for an $S$-proper poset such that 
$$\Vdash_{\bbP}``\dot{\bbQ}\text{ strongly preserves the tightness of }\calG".$$
Then $\bbP*\dot{\bbQ}$ (is $S$-proper and) strongly preserves the tightness of $\calG$. Moreover, the following holds: If $g\in\calG$, $M\prec H(\theta)$ is countable with $M\cap\omega_1\in S$, $\calG, \bbP, \dot{Q}$ and $p\in \bbP\cap M$ is an $(M,\bbP, \calG,g)$-generic condition, $\dot{q}$ is a $\bbP$-name for an element of $\dot{\bbQ}$ from $M$ such that 
$$p\Vdash ``\dot{q}\hbox{ is an }(M[G],\dot{\bbQ},\calG,g)\text{-generic condition}",$$
then $(p,\dot{q})$ is $(M,\bbP*\dot{\bbQ},\calG, g)$-generic condition.
\end{lem}

Preservation at limit stages is taken care of by the following lemma.

\begin{lem}(\cite[Proposition 6.4]{GHT}) Let $\calG$ be a tight cofinitary group. Let $\calP=\langle \bbP_\alpha, \dot{\bbQ}_\beta: \alpha\leq \gamma, \beta  <\gamma\rangle$ be a countable support iteration of $S$-proper forcing notions such that 
$$\Vdash_{\bbP_\alpha}``\dot{\bbQ}_\alpha\text{ strongly preserves the tightness of }\calG".$$
Let $M\prec H(\theta)$ be countable for $\theta$ sufficiently large regular cardinal such that $\calG,\calP,\gamma$ are elements of $M$, $M\cap \omega_1\in S$ and let $g\in\calG$ diagonalize every element of $\calI_i(\calG)^+\cap M$. Then, for every $\alpha\in M\cap\gamma$ and every $(M,\bbP_\alpha,\calG,g)$-generic condition $p\in\bbP_\alpha$ the following holds: 

If $\dot{q}$ is a $\bbP_\alpha$-name such that $p\Vdash_{\bbP_\alpha} ``\dot{q}\in\bbP_\gamma\cap M\text{ and }\dot{q}\upharpoonright \alpha\in \dot{G}_\alpha"$, 
then there is an $(M,\bbP_\alpha,\calG,g)$-generic condition $\bar{p}\in\bbP_\gamma$ such that $\bar{p}\upharpoonright \alpha =p$ and $\bar{p}\Vdash_{\bbP_\gamma}``\dot{q}\in \dot{G}"$.    
\end{lem}

The above two lemmas imply:

\begin{cor}(see \cite[Corollary 6.5]{GHT})\label{crX} Let $\calG$ be a tight cofinitary group. Let
$$\calP=\langle \bbP_\alpha, \dot{\bbQ}_\beta:\alpha\leq\gamma,\beta<\gamma\rangle$$ 
be a countable suppor iteration of $S$-proper forcing notions such that for each $\alpha<\gamma$
$$\Vdash_{\bbP_\alpha}``\dot{\bbQ}_\alpha\text{ strongly preserves the tightness of }\calG".$$  
Then $\Vdash_{\bbP_\gamma}``\calG\text{ is a tight cofinitary group}"$.
\end{cor}

Equipped with the above preservation theorems, we can prove our second main theorem:

\begin{thm}\label{main_two} It is relatively consistent that $\mathfrak{a}_g=\mathfrak{d}=\aleph_1<\mathfrak{c}=\aleph_2$, there is a $\Delta^1_3$-well-order of the reals and a co-analytic tight cofinitary group of cardinality $\aleph_1$.
\end{thm}
\begin{proof} Assume $V=L$. Let $\calG$ be the co-analytic tight cofinitary group constructed in the previous section. Note that $\calG$ provably consists only of constructible reals. 

For a given countable transitive model $\bar{M}$ of a sufficiently large fragment of $ZFC$, let $\chi(\bar{M})$ denote the $\leq_L$-least $g\in\calG$, which diagonalizes all elements of $\bar{M}\cap\calI_i(\calG)^+$. As in the proof of \cite[Lemma 4.3]{JBVFCS} one can observe that the function $\chi:\bar{M}\mapsto \chi(\bar{M})$ is $\Delta^1_3$-definable, since $(\bar{M},g)\in\chi$ iff
\begin{enumerate}
\item $\bar{M}\vDash``\text{sufficiently large fragment of }ZFC"$ (which is arithmetic)
\item $g\in\calG$ (which is $\Pi^1_1$)
\item $g$ diagonalizes every $T\in\bar{M}\cap\calI_i(\calG)^+$ (which is $\Pi^1_2$)
\item  for every $h\in \calG $ satisfying the above item, $g\leq_L h$ (which is $\Pi^1_2$).
\end{enumerate}    
Note that the image of every real under the Mostowski collapse is itself and so, we might refer to $\chi(M)$ 
instead of $\chi(\bar{M})$ for arbitrary countable models $M$. 

Let $\bbP^*=\langle \bbP_\alpha,\dot{\bbQ}_\beta:\alpha\leq\omega_2, \beta<\omega_2\rangle$ be the poset defined in Section 5.1. producing a cardinal preserving generic extension of the negation of CH and a $\Delta^1_3$-wellorder of the reals.  By induction on $\gamma<\omega_2$, we will prove that if $M$ is a countable elementary submodel of $L_{\omega_2}$ such that $M\cap\omega_1\in S$, $p\in \bbP_\gamma$, $\gamma,\bbP_\gamma, p$ are elements of $M$, then there is an $(M,\bbP_\gamma,\calG,\chi(M))$-generic condition $q$ extending $p$. This implies that for each $\gamma<\omega_2$, 
$\Vdash_{\bbP_\gamma}``\calG\text{ is a tight cofinitary group}"$ and so by Corollary \ref{crX} we obtain $L^{\bbP_{\omega_2}}\vDash(\calG\text{ is tight}$).

It is sufficient to consider iterands of the form $C(Y)$. Thus, assume the statement is proved for $\bbP_\gamma$ and let $\Vdash_{\bbP_\gamma}\dot{\bbQ}_\gamma=\dot{C}(Y)$. Let $M$ be a countable elementary submodel of $L_{\omega_2}$,  $M\cap \omega_1=\delta\in S$, $(p,\dot{q})\in \bbP_\gamma *\dot{C}(Y)$ with all relevant parameters in $M$.  Denote by $\bar{M}$ the Mostowski collapse of $M$ and note that $\bar{M}=L_\alpha$ for some $\alpha$. For each $x\in M$, let $\bar{x}$ denote the image of $x$ in $\bar{M}$ under the Mostowski collapse and let $g=\chi(\bar{M})$. By the inductive hypothesis, there is $q^*\leq p$  which is $(M,\bbP_\gamma,\calG,g)$-generic. Let $G$ be a $\bbP_\gamma$ generic filter of $L$ such that $q^*\in G$. Note that $L[G]=L[Y]$ for some $Y\subseteq\omega_1$ and in $L[Y]$ we have $\bar{M}[Y\cap \delta]\prec L_{\omega_2}[Y]$ and $g$ densely diagonalizes every element of $\calI_i(\calG)^+\cap\bar{M}[Y\cap \delta]$. Moreover, $\omega_1\cap\bar{M}[Y\cap\delta]=\delta\in S$ and both $\bar{M}$, and $\bar{M}[Y\cap \delta]$ are elements of $\calA_\delta=L_{\mu_\delta}[Y\cap \delta]$. By the $\Delta^1_3$-correctness of $\calA_\delta$, we get that $g\in\calA_\delta$. 

Fix a countable, cofinal in $\delta$ sequence $\{\delta_n:n\in\omega\}\in\calA_\delta$ such that for each $n\in\omega$, $\delta_n\in\bar{M}[Y\cap\delta]$. Clearly, the sequence itself is not an element of $\bar{M}[Y\cap\delta]$, however it is an element of $\calA_\delta$. Let $q\in M[Y]$ be the evaluation of $\dot{q}$ in $M[Y]$ and let $\bar{q}$ be its image in $\bar{M}[Y\cap\delta]$. We will find an extension $q_\omega\leq \bar{q}$ which is  $(M[Y], C(Y),\calG,g)$-generic condition and is an \underline{element of $\calA_\delta$}.  

Let $\{D_k:k\in\omega\}$ be an enumeration in $\calA_\delta$ of all dense subsets of $\bar{C(Y)}$ which are elements of $\bar{M}[Y\cap \delta]$ and let $\{\dot{T}_k:k\in\omega\}$ enumerate in $\calA_\delta$ all $\bar{C(Y)}$-names for subtrees of $\omega^{<\omega}$ in $\bar{M}[Y\cap\delta]$ which are forced to be in $\calI_i(\calG)^+$, so that each name appears cofinally often. Moreover fix a bijection $\varphi_\omega\to\omega^2$ such that $\varphi\in\calA_\delta$, e.g. $\varphi$ is a computable map, with coordinate maps $\varphi_0$ and $\varphi_1$ as a book-keeping device. We will inductively construct sequences $\langle q_n: n\in\omega\rangle, \langle \dot t_n : n\in\omega\rangle, \langle \dot k_n : n\in\omega\rangle$ in  $\mathcal A_\delta$ with the following properties:
    \begin{enumerate}
        \item $q_0 = \bar{q}$,
        \item $q_{n+1} \in C(Y) \cap \bar M[Y\cap\delta]$ and $q_{n+1} \leq_{n+1} q_n $
        \item\label{cond.codes} $\lvert q_{n+1} \rvert \geq \delta_n$,
        \item\label{cond.dense} $q_{n+1} \forces \bar D_n \cap \dot G  \neq \emptyset$,
        \item\label{cond.trees1} $\dot t_n$ is a $C(Y)^{\bar M[Y\cap\delta]}$-name in $\bar M[Y\cap\delta]$ for an element of $\dot T_{\phi_0(n)}$ and $\dot k_n$ is a $C(Y)^{\bar M[Y\cap\delta]}$-name in $\bar M[Y\cap\delta]$ for an element of $\omega$,
        \item\label{cond.trees2} $q_{n+1} \forces$``with $\dot s$ the $\phi_1(n)$-th node in $\dot T_{\phi_0(n)}$, $\dot k_n \in \dom(\dot t_n) \setminus \dom(\dot s)$ and $\dot t_n(\dot k_n) = \check g(\dot k_n)$''.
    \end{enumerate}
\noindent    
Suppose we have already constructed this sequence up to $q_n$. Let $q_{n+1} \leq_{n+1} q_n$ be the least condition (in the canonical well-ordering of $\mathcal A_\delta$) such that for each $(n+1)$-th splitting node $t$ of $q_n$, $q:=(q_{n+1})_t$ satisfies the following:
    \begin{enumerate}[(i)]
    \item\label{i.codes} $\lvert q\rvert \geq \delta_n$,
    \item\label{i.dense} $q\in \bar D_n$ and moreover, 
    \item \label{i.trees}
    for some $s,t \in 2^{<\omega}$ with $s\subseteq t$ and some $k \in \omega$, it holds that $q \forces$``$
    \check s$ is the $\phi_1(n)$-th node in $\dot T_{\phi_0(n)}$,
    $\check t \in T_{\phi_0(n)}$'', 
    and $k \in \dom(t) \setminus \dom(s)$ and $t(k) = g(k)$.
    \end{enumerate}

\noindent
It is clear that the set of $q$ satisfying \eqref{i.codes} and \eqref{i.dense} is dense in $C(Y)\cap \bar M[Y\cap\delta]$, as  $\bar{M}[Y\cap \delta]\prec L_{\omega_2}[Y]$. To see that a condition $q$ as required exists, it therefore remains to show the following claim:

\begin{clm*}
The set of conditions $q$ satisfying \eqref{i.trees} is dense in $C(Y)$.
\end{clm*}
\begin{proof}To see this claim, let us write $\dot T = \dot T_{\phi_0(n)}$ and let $q^* \in C(Y)$ be arbitrary. 
Find $q' \leq q^*$ and $s$ such that $q' \forces``\check s\text{ is the }\phi_1(n)\text{-th node in }\dot T"$. Note that $T := \{t \in \omega^{<\omega} \colon q' \not\forces \check t \notin \dot\dot T\}$
is an element of $\mathcal I_i(\calG)^+\cap \bar M[Y\cap \delta]$, as can be verified in a straightforward manner from the definition. Therefore, by assumption, $g$ densely diagonalizes $T$ and we can find $k$ and $t \in T$ such that $s \subseteq t$, $k \in \dom(t)\setminus \dom(s)$, and $t(k) = g(k)$.
Finally, as $t \in T$, we can find $q \leq q'$ such that  $q \forces \check t \in \dot T$.
This finishes the proof of the claim and hence the construction of $q_{n+1}$.
\end{proof}
    
It is clear from item \eqref{i.dense} that $q_{n+1}$ satisfies \eqref{cond.dense}. It is also clear that from item \eqref{i.trees}  that we can find names $\dot k_n, \dot t_n \in \bar M[Y\cap\delta]$ satisfying \eqref{cond.trees1} and \eqref{cond.trees2} (in fact, these names can be chosen to be finite).
Finally, since $\mathcal A_\delta$ satisfies $\ZF^-$ and all the required data for the definition of these sequences is an element of $\mathcal A_\delta$, the above definition yields sequences which are also elements of this model, as required.

Define $q_\omega = \bigcap_{n\in\omega} q_n$. To see that $q_\omega \in C(Y)$, observe that $\lvert q_\omega \rvert = \delta$ and that $q_\omega$ codes $Y$ up to $\delta$ since by \eqref{cond.codes}, $q_\omega$ codes $Y$ below $\delta$ 
    and because
    $q_\omega \in \mathcal A_\delta$.
    It is now straightforward to verify from the definitions that 
    $q_\omega$ is $(M[Y],C(Y), \calG,g)$-generic condition.
    \end{proof}

Our proof in fact shows that any sufficiently definable tight cofinitary group will survive the poset adjoining a 
$\Delta^1_3$-wellorder to the reals, as long as the elementarity (1).(b) requirement from Definition \ref{DFN_Sacks_Coding} accounts for the appropriate level of projective complexity. By Remark \ref{main_two} the witness to $\mathfrak{a}_g$ provided in the above theorem is of optimal projective complexity.

\section{Further applications}

Our construction provides not only a new proof of the existence of a Miller indestructible maximal cofinitary group (originally proved by Kastermans and Zhang with the use of a diamond sequence, see~\cite{BKYZ}) and answers Question 2 of~\cite{JBVFCS}, but gives a uniform proof of the existence of a co-analytic witness to $\mathfrak{a}_g$ in various forcing extensions:

\begin{cor}\label{COR_Main} Each of the following cardinal characteristics constellations is consistent with the existence of a co-analytic tight witness to $\mathfrak{a}_g$ and a $\Delta^1_3$-well-order of the reals: 
\begin{enumerate}
    \item $\mathfrak{a}_g=\mathfrak{u}=\mathfrak{i}=\aleph_1<\mathfrak{c}=\aleph_2$,
    \item $\mathfrak{a}_g=\mathfrak{u}=\aleph_1<\mathfrak{i}=\mathfrak{c}=\aleph_2$,
    \item $\mathfrak{a}_g=\mathfrak{i}=\aleph_1<\mathfrak{u}=\mathfrak{c}=\aleph_2$,
    \item $\mathfrak{a}_g=\aleph_1<\mathfrak{i}=\mathfrak{u}=\aleph_2$.
\end{enumerate}
In addition, in each of the above constellations the characteristics $\mathfrak{a}$, $\mathfrak{a}_e$, $\mathfrak{a}_e$, $\mathfrak{a}_p$ can have tight co-analytic witnesses of cardinality $\aleph_1$; in items $(1)$ and $(2)$, the ultrafilter number $\mathfrak{u}$ can be witnessed by a co-analytic ultrafilter base for a $p$-point; in items $(1)$ and $(3)$ the independence number can be witnessed by a co-analytic selective independent family.
\end{cor}
\begin{proof} Work over $L$ and proceed with a countable support iteration as in~\cite[Theorem 6.1]{JBVFCS} for item $(1)$ and as in~\cite[Theorem 6.2]{JBVFCS} for items $(2)-(4)$. 
\end{proof}

\section{questions}

As discussed in the second section, the interaction between maximal cofinitary groups and maximal eventually different families of permutations is not really clear.
To the knowledge of the authors, the following is not even known:

\begin{question}
    Is there a maximal cofinitary group that is not a maximal eventually different family of permutations?
\end{question}

As noted before, every maximal eventually different family of permutations, which happens to be a group, is in fact a maximal cofinitary group.
We expect the answer to be yes, otherwise we would immediately obtain $\mathfrak{a}_p \leq \mathfrak{a}_g$.
Note, that a tight cofinitary group in the sense of this paper is also a witness for $\mathfrak{a}_p$, so the results of this paper are not able to be used to separate $\mathfrak{a}_g$ and $\mathfrak{a}_p$.
Moreover, no relations or the absence thereof are known between $\mathfrak{a}_e, \mathfrak{a}_p$ and $\mathfrak{a}_g$.

\begin{question}
    Can any of $\mathfrak{a}_e$, $\mathfrak{a}_p$ and $\mathfrak{a}_g$ consistently be separated from another?
\end{question}

Finally, T\"ornquist proved for mad families that the existence of a $\Sigma^1_2$ mad family already implies the existence of a $\Pi^1_1$ mad family of the same size \cite{AT}. Similar results have been obtained for maximal eventually different families, maximal independent families and towers, see respectively \cite{VFDS}, \cite{JBVFYK} and \cite{VFJS}. 
We ask if the same is true for maximal cofinitary groups.
Note, that unlike as for mad families, this question is only really interesting in models of {$\neg$\sf CH}, as there always is a Borel maximal cofinitary group, which has to be of size continuum \cite{HS1}.

\begin{question}
   Does the existence of a $\Sigma^1_2$ maximal cofinitary group imply the existence of a $\Pi^1_1$ maximal cofinitary group of the same size?
\end{question}

\end{document}